\documentclass[11pt]{amsart}

\usepackage[margin=1in]{geometry}
\usepackage{soul}
\usepackage{leftindex}
\usepackage[utf8]{inputenc}
\usepackage{csquotes}

\usepackage[utf8]{inputenc}
\usepackage{dsfont}
\usepackage{amsmath}
\usepackage{amsfonts}
\usepackage{amssymb}
\usepackage{amsthm}
\usepackage{xcolor}
\usepackage{graphicx}
\usepackage{hyperref}

\newtheorem{thm}{Theorem}[section]

\newtheorem{prop}[thm]{Proposition}

\newtheorem{lem}[thm]{Lemma}
\newtheorem{defn}[thm]{Definition}
\newtheorem*{question*}{Question}

\newtheorem{cor}[thm]{Corollary}
\newtheorem{claim}{Claim}[thm]
\newtheorem*{claim*}{Claim}

\newtheorem{Main}{Theorem}

\newtheorem{prop*}[Main]{Proposition}
  
\newtheorem{fact}{Fact}
\numberwithin{equation}{section}

\newcommand{\Ex}{\mathrm{Ex}}
\newcommand{\PSL}{\mathrm{PSL}}
\newcommand{\Z}{{\mathbb{Z}}}
\newcommand{\N}{\mathbb{N}}
\newcommand{\R}{\mathbb{R}}

\newcommand{\Joi}{\mathrm{Joi}}
\newcommand{\disth}{\mathrm{d}_H}
\newcommand{\dist}{\mathrm{d}_X}
\newcommand{\fin}{\mathrm{fin}}
\newcommand{\K}{\mathcal{K}}
\newcommand{\ud}{\bar d}
\newcommand{\eps}{\varepsilon}
\newcommand{\be}{\begin{equation}}
\newcommand{\ee}{\end{equation}}
\newcommand{\cP}{\mathcal{P}}

\usepackage[draft]{todonotes}

\usepackage[backend=biber, style=alphabetic]{biblatex}
\title[IC-rigidity for abelian actions]{IC-rigidity for abelian actions}
\author[A. Kanigowski, K. Kasprzak, and M. Lorenzo-Laguno]{Adam Kanigowski, Kosma Kasprzak, and Miriam Lorenzo-Laguno}
\begin{document}
\begin{abstract}
    Let $G$ be a group acting continuously on a compact metric space $X$. This naturally induces an action of $G$ on the space $\K(X)$ of nonempty compact subsets of $X$. In \cite{KraSchmieding2026}, Kra and Schmieding define a system $(X,G)$ to be \textit{IC-rigid} if the union of the supports of the $G$-invariant Borel probability measures on $\K(X)$ is as small as possible. They leave open the question of existence of IC-rigid abelian systems.
    
    We show the existence of IC-rigid actions for three classes of abelian groups: mixing rank-one $\Z$-actions, horocycle flows, and certain mixing actions on a Cantor set of $G=\bigoplus_{n=1}^{\infty}G_n$, where $(G_n)_n$ is any sequence of nontrivial finite abelian groups. In the course of proving these results, we introduce the notion of \textit{weakly asymptotic pairs} and establish its close connection to topological minimal self-joinings.
\end{abstract}
\maketitle
\section{Introduction}
\begingroup
\renewcommand{\thefootnote}{}
\footnotetext{No generative AI tools were used in producing the ideas, proofs, or text appearing in this manuscript.}
\endgroup
Consider a locally compact, second countable amenable group $G$ acting continuously on a compact metric space $X$, and let $\K(X)$ denote the space of nonempty compact subsets of $X$, endowed with the Hausdorff metric. The induced action of $G$ on $\K(X)$ has long been a topic of study as a tool to further understand the topological properties of the underlying system $(X,G)$ \cite{BauerSigmund1975,Glasner1979,GlasnerWeiss1995,Banks2005,LiOprochaYeZhang2017}. In \cite{KraSchmieding2026}, Kra and Schmieding introduce a new perspective on this setting by studying the distribution of the $G$-invariant probability measures on $\K(X)$. 

More precisely, Kra and Schmieding consider the space of Borel probability measures on $\K(X)$ invariant under the induced action of a discrete group $G$. They refer to these measures as \textit{invariant random compacts} (IRCs). Every such system admits at least one IRC, namely $\delta_X$. Moreover, when $G$ preserves a measure $\mu$ on $X$, one can construct further IRCs supported on finite subsets of $X$ by pushing forward the product measure $\mu^{\times k}$ under suitable maps from $X^{\times k}$ to $\K(X)$ (see Section \ref{sec:IC} for details). Kra and Schmieding denote by $\K_{\text{fin}}(X)\subset \K(X)$ the set of all finite sets of $X$, and call \textit{finitary} any IRC fully supported on $\K_{\text{fin}}(X)$. 

Motivated by the examples above, they define a system to be \textit{IC-rigid} if every IRC decomposes into a finitary component and a $\delta_X$ component. 
Equivalently, every IRC is supported on a subset of $\K_{\fin}(X)\cup \{X\}$, so the union of the supports of all IRCs is as small as possible.

Here, we answer one of the questions posed in \cite{KraSchmieding2026}: 
\begin{question*}
    Does there exist an IC-rigid action of an abelian group?
\end{question*} 
Our answer is positive. In fact, we provide three families of examples of IC-rigid abelian actions. 

First, we show existence of IC-rigid $\Z$-actions:
\begin{Main}\label{mainA}
    All mixing rank-one $\Z$-actions are IC-rigid.
\end{Main}

We also consider actions of countable direct sums of finite groups introduced by Danilenko in \cite{Danilenko_2001} and studied further in \cite{Danilenko_2006}. They satisfy a natural extension of the rank-one condition to general locally compact second countable amenable groups, sometimes referred as \emph{funny} rank-one, introduced by J.P. Thouvenot (see \cite{Ferenczi1985}) and generalized in \cite{Sokhet1997} (see also \cite{delJuncoYassawi2000}). Under the additional assumption of abelianity, we prove that these actions are IC-rigid:
\begin{Main}\label{mainF}
Let $(G_n)$ be a sequence of nontrivial finite abelian groups, and let 
\[
G=\bigoplus_{n=1}^{\infty}G_n.
\]
The group $G$ admits $IC$-rigid actions.
\end{Main}

Additionally, we extend the notion of IC-rigidity in \cite{KraSchmieding2026} to non-discrete groups. For a general measure-preserving system $(X,G,\mu)$, we denote by $\K_G(X)\subset\K(X)$ the collection of finite unions of compact orbit segments under the action of the centralizer of $G$. For every $S\in\K_G(X)$, one can then construct an IRC by pushing forward the product measure $\mu^{\times k}$, for some $k\geq 1$, under a map $X^{\times k}\to\K(X)$ such that $S$ lies on the support of the resulting measure. In this setting, IC-rigidity is equivalent to the condition that every IRC is supported on $\K_G(X)\cup \{X\}$. 

We then find IC-rigid $\R$-actions:
\begin{Main}\label{mainB}
    Let $\Gamma$ be a maximal nonarithmetic cocompact lattice in $\PSL_2(\mathbb{R})$, let $\mu$ denote the Haar measure on $X=\PSL_2(\mathbb{R})/\Gamma$, and let $h_t$ denote the horocycle flow in $X$. Then $(X,h_t,\mu)$ is IC-rigid.
\end{Main}

To prove the preceding two theorems, we introduce the notion of \textit{weakly asymptotic pairs} (see Section \ref{sec:WP}), and present a setting in which it is key to showing IC-rigidity. 

\begin{Main}\label{mainC}
    Let $X$ be a compact metric space and let $\mu$ be a probability measure on $X$ invariant under the action of a group $G$. Suppose $(X,G,\mu)$ is minimal, uniquely ergodic, mixing of all orders, has minimal self-joinings of all orders, and has no weakly asymptotic pairs. Moreover, assume that every point $(x_1,\cdots,x_N)\in X^{\times N}$ is generic for an $(X^{\times N},\{g^{\times N}\}_{g\in G})$-ergodic measure. Then $(X,G)$ is IC-rigid. 
\end{Main}

Under the additional assumption that $G$ is countable and abelian, the conditions of Theorem \ref{mainC} can be presented in a simpler equivalent form, resulting in the following theorem.

\begin{Main}\label{mainD}
    Let $X$ be a compact metric space and let $\mu$ be a probability measure on $X$ invariant under the action of a countable abelian group $G$. Suppose $(X,G,\mu)$ is minimal, uniquely ergodic, mixing of all orders, has minimal self-joinings of all orders, and has no weakly asymptotic pairs. Then $(X,G)$ is IC-rigid.
\end{Main}

To prove Theorems \ref{mainB} and \ref{mainF}, we show that the systems satisfy the hypotheses of Theorems \ref{mainC} and \ref{mainD}, respectively. Interestingly, in the discrete abelian case, this method requires the group $G$ to not be finitely generated:
\begin{Main} \label{main:Z}
    For finitely generated abelian groups $G$, there are no infinite systems satisfying the conditions of Theorem \ref{mainD}.
\end{Main}

To prove Theorem \ref{main:Z}, we first show that any system satisfying the hypotheses of Theorems \ref{mainC} or \ref{mainD} has topological minimal self-joinings of all orders.
We then use a result of Bitar, Donoso, and Petite \cite{BitarDonosoPetite2026}, which states that no infinite action of a finitely generated abelian group has topological minimal self-joinings of all orders. This generalizes a result by King \cite{King_1990} for $\Z$-actions.

Note that, in particular, the systems of Theorems \ref{mainF} and \ref{mainB} provide examples of abelian actions with topological minimal self-joinings of all orders.

Theorem \ref{mainA} cannot follow directly from Theorem \ref{mainD} in light of Theorem \ref{main:Z}. Indeed, rank-one $\mathbb{Z}$-actions have a fixed point, so they are never uniquely ergodic. Our proof of Theorem \ref{mainA}, although specialized to the specific setting of rank-one systems, is still quite similar in spirit to the proof of Theorem \ref{mainD}. It relies heavily on mixing of all orders of mixing rank-one systems, shown by Kalikow in \cite{Kalikow1984}, and ideas from King's proof of the minimal self-joinings property for this class in \cite{King1988}.

\subsection*{Plan of the paper}

We begin by establishing our setting in Section \ref{sec:prelim}. In Subsections \ref{sec:em} and \ref{sec:joi}, we recall several classical definitions and results concerning ergodicity, mixing, and joinings. In Subsection \ref{sec:IC}, we introduce the notion of IC-rigidity, adapting the definition to apply to both discrete and non-discrete groups. We also prove Proposition \ref{prop:IC}, which provides a characterization of IC-rigidity.

In Section \ref{sec:combi}, before turning to the proofs of the main theorems, we prove three combinatorial results: Lemmas \ref{lemma:countprod}, \ref{lemma:countoffd}, and \ref{lemma:countdisc}, that simplify the arguments throughout the remainder of the paper.

We then proceed with the proofs of the main theorems. In Section \ref{sec:rk1}, we introduce mixing rank-one systems and prove Theorem \ref{mainA}. In Section \ref{sec:WP}, we define weakly asymptotic pairs and prove Theorem \ref{mainC}. In section \ref{sec:WP2} we consider countable abelian groups $G$, and prove Theorems \ref{mainD} and \ref{main:Z}. In Section \ref{sec:hf}, we draw on classical results about the horocycle flow to verify the hypotheses of Theorem \ref{mainD}, hence proving Theorem \ref{mainC}. Finally, in Section \ref{sec:dani}, we recall the actions constructed by Danilenko and prove that they satisfy the assumptions of Theorem \ref{mainD}, confirming Theorem \ref{mainF}.

\section{Preliminaries}\label{sec:prelim}

\subsection{Ergodicity and mixing} \label{sec:em}
Let $G$ be a locally compact second countable amenable group acting continuously on a compact metric space $X$. We recall some classical definitions.

\begin{defn}[Ergodic measure]Let $\mu$ be a measure on $X$ invariant under the action of $G$. We say $\mu$ is \textit{ergodic} if every $G$-invariant Borel set $A\subset X$ satisfies $\mu(A)\in\{0,1\}$.
\end{defn}

\begin{defn}[Minimality] The action of $G$ on $X$ is \textit{minimal} if every orbit is dense in $X$.
\end{defn}

\begin{defn}[Unique ergodicity] The action of $G$ on $X$ is \textit{uniquely ergodic} if there exists a unique $G$-invariant Borel probability measure on $X$.
\end{defn}

\begin{defn}[Mixing] Let $\mu$ be a $G$-invariant probability measure on $X$. The system $(X,G,\mu)$ is \textit{mixing} if for every pair $A,B\subset X$ of Borel sets, \[\lim_{g\to\infty}\mu(A\cap gB)=\mu(A)\mu(B).\]
\end{defn}

\begin{defn}[Mixing of all orders] Let $(X,G,\mu)$ be an measure-preserving system. It is \textit{mixing of all orders} if for any $N\geq 2$ and every tuple $A_0,\ldots,A_{N}\subset X$ of Borel sets, \[\lim_{g_i^{-1}g_j\to\infty}\mu(g_0A_0\cap\cdots\cap g_nA_N)=\mu(A_0)\cdots\mu(A_N).\]
\end{defn}

We will use the following three classical facts:
\begin{fact}[Ergodic decomposition theorem]
    Every $G$-invariant Borel probability measure on $X$ is the barycenter of a probability measure supported on the set of ergodic $G$-invariant measures.
\end{fact}

\begin{fact}[Birkhoff's pointwise ergodic theorem]
    Every ergodic probability measure $\mu$ has a generic point. That is, letting $(F_n)$ be a tempered Følner sequence in $G$ and $m$ be the Haar measure on $G$, there exists $x\in X$ such that for any continuity set $A\subset X$
    \[
    \lim_{n\to\infty}\frac{1}{m(F_n)}\int_{F_n}\delta_A(g x)\,dm(g)=\mu(A).
    \]
\end{fact}

\begin{fact}[Positive measure for open sets]
    If $(X,G,\mu)$ is minimal and uniquely ergodic, every open set $A\subset X$ has $\mu(A)>0$.
\end{fact}

\subsection{Joinings} \label{sec:joi}

Let $(X, G, \mu)$ be a measure-preserving system. We introduce some definitions adapted from \cite{King1988}.
 
\begin{defn}[Self-joining] Given $N\geq 2$, an \textit{$N$-fold self-joining} is a Borel probability measure on $X^{\times N}$ invariant under $\{g^{\times N}\}_{g\in G}$, for which $\mu$ is the marginal under the natural projection onto each of the coordinates. 
\end{defn}
    
    For $N\in\mathbb{Z}^+$ we denote by $\Joi_N(G,\mu)$ the family of all $N$-fold self-joinings of the system $(X, G, \mu)$. Let $\prescript{N}{}{\Delta}$ be the diagonal joining:
    \[
    \prescript{N}{}{\Delta}(A_1\times\ldots\times A_N)=\mu(A_1\cap \ldots \cap A_N),
    \]
    where $A_i\subset X$, $i=1,\ldots,N$, are any Borel sets.
    For a sequence $\textbf{g}=(g_1, \ldots, g_{N})\subset G$  we denote
    \[
    \prescript{N}{}{\Delta}^\textbf{g}=\left( g_1\times\ldots \times g_{N}\right)_* \prescript{N}{}{\Delta}.
    \]
    If $g_i^{-1}g_j$ lies in the center of $G$, $Z(G)$, for every $1\leq i,j\leq N$, we call such $\prescript{N}{}{\Delta}^\textbf{g}$ an off-diagonal joining, and {if the terms of $\textbf{g}$ are distinct we call $\prescript{N}{}{\Delta}^\textbf{g}$ properly off-diagonal}.

    We say an $N$-fold self-joining $\nu$ is trivial, if it is (possibly after permuting indices) a direct product of off-diagonal joinings. If these joinings are properly off-diagonal, we additionally call $\nu$ non-degenerate.

    For each $N\in \Z^+$, we denote by $\mu^{\times N}\in \Joi_N(G,\mu)$ the trivial non-degenerate joining given by
    \[
    \mu^{\times N}(A_1\times\ldots\times A_N)=\mu(A_1)\cdots \mu( A_N).
    \]

\begin{defn}[Minimal self-joinings] Let $(X,G,\mu)$ be a measure-preserving system. Given $N\geq 2$, we say it has \textit{N-fold minimal self-joinings (MSJ)} if all the $(X^{\times N},\{g^{\times N}\}_{g\in G})$-ergodic elements of $\Joi_N(G,\mu)$ are trivial. 

If a system has $N$-fold MSJ for every $N\geq 2$, we say it has MSJ of all orders.
\end{defn}

\subsection{IC-rigidity} \label{sec:IC}
Let $G$ be a locally compact second countable amenable group acting continuously on a compact metric space $X$ with distance $\dist$, and let $\mu$ be a $G$-invariant probability measure on $X$. Fix a tempered Følner sequence $(F_n)$ in $G$, let $Z(G)$ denote the center of $G$, and fix a tempered Følner sequence $(F^Z_n)$ in $Z(G)$. 

We refer to the notion of IC-rigidity introduced by Kra and Schmieding in \cite{KraSchmieding2026}. 

For any topological space $Y$, denote by 
\[\K(Y)=\{K\subset Y:\ K\neq \emptyset\text{ compact}\},\] 
endowed with the Hausdorff measure \[\disth(A,B)=\max\left\{\max_{x\in A}\min_{y\in B}\text{d}_Y(x,y),\ \max_{y\in B}\min_{x\in A}\text{d}_Y(x,y)\right\},\] whenever $Y$ is a metric space, and by 
\[\K_{\fin}(Y)=\{K\in \K(Y):\# K<\infty\}.\]
For $G$ acting on $X$, denote by
\[
\K_{G}(X)=\left\{\bigcup_{x\in K}\{gx:g\in I_x\}:K\in\K_{\fin}(X),\, I_x\in \K(Z(G)) \,{\text{for all }} x\in K\right\}.
\]
Clearly $\{X\}$, $\K_{\fin}(X)$, and $\K_{G}(X)$ are invariant under the action induced by $G$ on $\K(X)$. 

\begin{defn}\label{def:ICrigid} We say that $(X,G)$ is \textit{IC-rigid} if every Borel probability measure $\nu$ on $\K(X)$ invariant under the action induced by $G$ satisfies $\nu(\K_{G}(X)\cup \{X\})=1$.
\end{defn}

This definition extends the one in \cite{KraSchmieding2026} so that the notion of IC-rigidity remains nontrivial for actions of non-discrete groups, such as flows. Indeed, for any discrete group $G$, every compact subset is finite, so $\K_G(X)=\K_{\fin}(X)$, and hence the two are equivalent.

To illustrate the choice of $\K_G(X)$ in the definition above, note that for every $n\in\Z^+$ and $I_1,\dots,I_n\in\K(Z(G))$, the map \[\phi:X^{\times n}\to\K(X),\qquad (x_1,\dots,x_n)\mapsto \bigcup_{i=1}^n\{g x_i:g\in I_i\},\] induces an invariant measure $\phi_*\mu^{\times n}$ on $\K(X)$, under the action induced by $G$. Moreover, it is supported on \[\left\{\bigcup_{i=1}^n\{gx_i:g\in I_i\}:(x_1,\dots,x_n)\in X^{\times n}\right\}\subset\K_G(X).\] Therefore, we say that a system is IC-rigid when the invariant measures on $\K(G)$ are supported on the smallest possible subset. In particular, every proper compact invariant set of $(X,G)$ must be in $\K_G(X)$.

Let $\ud(D)=\limsup_{{n\to\infty}} m(F_n)^{-1}m(F_n\cap D)$ denote the upper density of a set $D\subset G$. 

We have the following characterization of IC-rigidity.

\begin{prop}\label{prop:IC}
A system $(X, G)$ is not IC-rigid if and only if there exists a compact set $A\subset X$ such that 
\begin{itemize}
    \item[a)] for every $\varepsilon>0$ the sets 
    \[
    D_{\eps} :=\{g\in G: g A \text{ is }\varepsilon\text{-dense in } X\}
    \] 
    satisfy $\lim_{\eps\to0}\ud(D_{\eps})=0.$
    \item[b)] if for every $M\in\mathbb{Z}^+$ and $\varepsilon>0$ we write 
    \[
    B(x,M,\eps):=\left\{y:\inf_{g\in F^Z_M}\dist(y,gx)<\eps\right\},
    \]
    then the sets 
    \[
    S_{M,\eps}:=
    \left\{g\in G: gA \subset\bigcup_{x\in K}B(x,M,\eps)\text{ for some } K\subset X \text{ with } \#K\leq M \right\}
    \] 
    satisfy $\lim_{\eps\to0}\ud(S_{M,\eps})=0.$
\end{itemize}
\end{prop}

\begin{proof}
    For each $\eps>0$ and $M\in\Z^+$, define the set
    \[H_M:=\left\{\bigcup_{x\in K}\{gx:g\in I_x\}: K\subset X,\,\#K\leq M, \text{ where }I_x\in\K(G),\,I_x\subset  F^Z_M \,{\text{ for all }} x\in K\right\},\]
    so that $ \K_{G}(X)=\bigcup_{M\in\Z^+}H_M$, and the open sets
    \[E_{\eps}:=\{B\in\K(X):\disth(B,\{X\})<\eps\},\qquad H_{M,\eps}:=\{B\in\K(X):\disth(B,H_M)<\eps\}.\]
    
    Fix $A\in\K(X)$ and let $D_{\eps}$ and $S_{M,\eps}$ be the sets defined above in a) and b) corresponding to $A$.

    Note that $\ud(D_\eps)$ and $\ud (S_{M,\eps})$ are decreasing in $\eps$, for $M$ fixed, and that
    \be\label{eq:auxIC}\begin{aligned}
    \{g\in G:gA\in E_\eps\}= D_\eps&\subset\{g\in G:gA\in{\overline{E_{\eps}}}\},
    \\
    \{g\in G:gA\in{H_{M,\eps}}\}= S_{M,\eps}&\subset\{g\in G:gA\in\overline{H_{M,\eps}}\},\end{aligned}
    \ee
    where the closures are in the context of the Hausdorff metric.
    Moreover, note that $E_{\eps},\overline{E_\eps}\to \{X\}$ and $H_{M,\eps},\overline{H_{M,\eps}}\to H_M$ as $\eps\to 0$.
 
    Suppose $(X,G)$ is not IC-rigid. Then there exists a Borel probability measure $\hat \nu$ on $\K(X)$ such that $\hat \nu(\K(X)\setminus (\K_{G}(X)\cup \{X\}))>0$. 
    Since $\K_{G}(X)\cup \{X\}$ is invariant under $G$, there exists an ergodic probability measure $\nu$ on $\K(X)$ supported on $\K(X)\setminus \K_{G}(X)$ with $\nu(\{X\})=0$, and thus a compact set $A\in \K(X)$ for which the measures
    \be\label{eq:conv1}
    \nu_n:=\frac{1}{m(F_n)}\int_{F_n}\delta_{\{gA\}}\, dm(g)
    \ee
     converge weakly$^*$ to $\nu$.
     Then, since $\nu$ is finite, by the Portmanteau theorem and \eqref{eq:auxIC}--\eqref{eq:conv1},
    \[
    \lim_{\eps\to 0}\ud(D_{\eps})=\lim_{\eps\to 0}\limsup_{n\to\infty}\frac{m(F_n\cap D_\eps)}{m(F_n)}\leq\lim_{\eps\to 0}\limsup_{n\to\infty}\nu_n(\overline{E_\eps})\leq\lim_{\eps\to 0}\nu(\overline{E_\eps})=\nu(\{X\})=0,
    \]\[
    \lim_{\eps\to 0}\ud(S_{M, \eps})=\lim_{\eps\to 0}\limsup_{n\to\infty}\frac{m(F_n\cap S_{M,\eps})}{m(F_n)}\leq\lim_{\eps\to 0}\limsup_{n\to\infty}\nu_n(\overline{H_{M,\eps}})\leq\lim_{\eps\to 0}\nu(\overline{H_{M,\eps}})=\nu(H_M)=0.
    \]
    Therefore $A$ satisfies both {a)} and {b)}.

    We now prove the converse. Assume there exists a compact set $A\in\K(X)$ that satisfies {a)} and {b)} and again let $\nu_n=m(F_n)^{-1}\int_{F_n}\delta_{\{gA\}}\, dm(g)$. By weak$^*$-compactness of the space of probability measures on $\K(X)$, there exists a subsequence $\{n_k\}$ along which $\nu_n$ converges weakly$^*$ to a probability measure $\nu$. Moreover, by construction $\nu$ is $G$-invariant. 
    Again, the Portmanteau theorem and \eqref{eq:auxIC} yield
    \[
    0=\lim_{\eps\to 0}\ud(D_{\eps})=\lim_{\eps\to 0}\limsup_{n\to\infty}\frac{m(F_n\cap D_\eps)}{m(F_n)}\geq\lim_{\eps\to 0}\liminf_{k\to\infty}\nu_{n_k}({E_{\eps}})\geq\lim_{\eps\to 0}\nu({E_{\eps}})=\nu(\{X\}),
    \]\[
    0=\lim_{\eps\to 0}\ud(S_{M, \eps})=\lim_{\eps\to 0}\limsup_{n\to\infty}\frac{m(F_n\cap S_{M,\eps})}{m(F_n)}\geq\lim_{\eps\to 0}\liminf_{k\to\infty}\nu_{n_k}({H_{M,\eps}})\geq\lim_{\eps\to 0}\nu({H_{M,\eps}})=\nu(H_M).
    \]
    Therefore $\nu(\K_{G}(X)\cup \{X\})=\nu(\bigcup_{M\in\Z^+} H_M\cup \{X\})=0$, so $(X,G)$ is not IC-rigid.
\end{proof}

\section{Auxiliary combinatorial lemmas}\label{sec:combi}
Throughout this section, let $G$ be a locally compact second countable amenable group acting continuously on a compact metric space $X$, let $\mu$ be a $G$-invariant probability measure on $X$, and fix a tempered Følner sequence $(F_n)$ in $G$. 

We establish three combinatorial lemmas that will be used repeatedly in the proofs of the main results.
\begin{lem}\label{lemma:countprod}
        Let $\mathcal{P}$ be a collection of finitely many Borel subsets of $X$ of positive measure. Then there exists a positive integer $N$ such that 
    \[
    \mu^{\times N}\biggl(\Big\{(x_1, \ldots, x_N)\in X^{\times N}: \{x_1, \ldots, x_N\} \text{ intersects all elements of $\mathcal{P}$}\Big\}\biggl)\,>\frac{3}{4}.
    \]
\end{lem}
\begin{proof}
    Let $\alpha=\min_{P\in \cP}\mu(P)$, which is positive by assumption, and for every $N\in \Z^+$ set
    \[
    B_N:=\left\{(x_1, \ldots, x_N)\in X^{\times N}: \{x_1, \ldots, x_N\} \text{ intersects all elements of $\mathcal{P}$}\right\}.
    \]
    Note $B_N=\bigcap_{P\in\cP}\bigcup_{i=1}^N X^{\times (i-1)}\times P\times X^ {\times(N-i)}$. Therefore for every $N\in\Z^+$
    \[\begin{aligned}
    \mu^{\times N}&(B_N)\geq 1- \mu^{\times N}\biggl(\bigcup_{P\in\cP}\bigcap_{i=1}^N X^{\times(i-1)}\times P^c\times X^{\times(N-i)}\biggr)
     \geq 1-|\cP|(1-\alpha)^N,
    \end{aligned}
    \]
    which is greater than $3/4$ for $N$ large enough (depending on $\alpha$ and $|\cP|$). 
\end{proof}
\begin{lem}\label{lemma:countoffd}
        Suppose $(X,G, \mu)$ is mixing of all orders, and let $\mathcal{P}$ be a collection of finitely many Borel subsets of $X$ of positive measure. Let $N$ be the integer given by Lemma \ref{lemma:countprod}. Then there exists a positive integer $L$ such that for any $\mathbf{g}=\left( g_1,\ldots ,g_{N}\right)\in G^{\times N}$ with $g_i^{-1}g_j\not\in F_L$ for every $1\leq i\neq j\leq N$, the measure $\prescript{N}{}{\Delta}^\mathbf{g}$ satisfies
    \[
    \prescript{N}{}{\Delta}^\mathbf{g}\biggl(\Big\{(x_1, \ldots, x_N)\in X^{\times N}: \{x_1, \ldots, x_N\} \text{ intersects all elements of $\mathcal{P}$}\Big\}\biggl)>\frac{1}{2}.
    \]
\end{lem}
\begin{proof}
    Set $B_N=\bigcap_{P\in\cP}\bigcup_{i=1}^N X^{\times (i-1)}\times P\times X^ {\times(N-i)}$. Since $(X,G, \mu)$ is mixing of all orders and $B_N\subset X^{\times N}$ is a finite union of rectangles, for any $\textbf{g}_n=(g_1^n,\cdots,g_N^n)$ with $\lim_{n\to\infty}(g_i^n)^{-1}g_{j}^n=\infty$, we have the convergence result
    \[\lim_{n\to\infty}\prescript{N}{}{\Delta}^{\textbf{g}_n}(B_N)=\mu^{\times N}(B_N).\]
    In particular, there exists $L\in\Z^+$ such that for any $\textbf{g}=(g_1,\cdots,g_N)$ with $g_i^{-1}g_{j}\not\in F_L$ for every $1\leq i\neq j\leq N$
    \[\left|\prescript{N}{}{\Delta}^{\textbf{g}}(B_N)-\mu^{\times N}(B_N)\right|<1/4\implies\prescript{N}{}{\Delta}^{\textbf{g}}(B_N)>1/2,\]
    from our choice of $N$.
\end{proof}

\begin{lem}\label{lemma:countdisc}
    Suppose $G$ is discrete. Let $(X,G, \mu)$ be a measure-preserving system which is mixing of all orders, and let $\mathcal{P}$ be a collection of finitely many Borel subsets of $X$ of positive measure. Then there exists a positive integer $K$ such that any non-degenerate trivial $K$-fold self-joining $\nu$ of $T$ satisfies
    \[
    \nu\biggl(\Big\{(x_1, \ldots, x_K)\in X^{\times K}: \{x_1, \ldots, x_K\} \text{ intersects all elements of $\mathcal{P}$}\Big\}\biggl)>\frac{1}{2}.
    \]
\end{lem}
\begin{proof}
    Fix $N$ and $L$ from Lemmas \ref{lemma:countprod} and \ref{lemma:countoffd}. Since $G$ is discrete, $|F_L|<\infty$, where $|\cdot|$ denotes the counting measure on $G$.

    Choose $K=N^2 (2|F_L|+1)$ and fix a non-degenerate trivial $K$-fold self-joining $\nu$ on $X^{\times K}$. If there exists a subset $\{i_1,\ldots,i_N\}\subset[1,K]$ of indices such that the projection of $\nu$ on these coordinates is equal to $\mu^{\times N}$, then
    \[\begin{aligned}
    \nu\left(B_K\right)\geq
    \nu\biggl(\Big\{(x_1, \ldots, x_K)\in X^{\times K}: \{x_{i_1}, \ldots, x_{i_N}\} \text{ intersects every  $P\in\mathcal{P}$}\Big\}\biggl)=\mu^{\times N}(B_N)>\frac{1}{2}.
    \end{aligned}
    \]
    
    Let $\tilde K=N(2|F_L|+1)$. If there does not exists a subset of indexes in $[1,K]$ such that the projection of $\nu$ on these coordinates is equal to $\mu^{\times N}$, then there exists a subset $\{i_1,\ldots,i_{\tilde K}\}\subset[1,K]$ such that the projection of $\nu$ on these coordinates is equal to $\prescript{\tilde K}{}{\Delta}^{\textbf{g}}$, for some vector $\textbf{g}=(g_1,\ldots,g_{\tilde K})$ of distinct elements of $G$. In particular, there exists $\{j_1,\ldots,j_{N}\}\subset\{1,\ldots,{\tilde K}\}$ such that $(g_{j_r})^{-1}g_{j_s}\not \in F_L$ for every $1\leq s\neq r\leq N$. Call $\mathbf{g'}=(g_{j_1},\ldots,g_{j_N})$. Then
    \[\begin{aligned}
    \nu\left(B_K\right)
    &\geq
    \prescript{\tilde K}{}{\Delta}^{\textbf{g}}(B_{\tilde K})\geq \prescript{N}{}{\Delta}^{\mathbf{g'}}(B_{N})>\frac{1}{2}.
    \end{aligned}
\]
    
\end{proof}

\section{Mixing rank-one $\Z$-actions are IC-rigid}\label{sec:rk1}
The goal of this section is to prove Theorem \ref{mainA}, which implies the existence of IC-rigid $\Z$-actions.

\subsection{Rank one maps}

We use the definition of rank-one maps as subshifts following \cite{Kalikow1984}.

For $x\in \{0, 1\}^{\mathbb{Z}}$ and nonnegative integers $a<b$ we write $x[a, b]=(x_a, x_{a+1}, \ldots, x_{b-1})$. For a set $A\subset \{0, 1\}^{\mathbb{N}}$ we write $A[a, b]=\{x[a, b]: x\in A\}$. 

Let $(q_n)$ be an infinite sequence of positive integers, and let $(a_{i, j})_{i\in\mathbb{N}, j\in \{1, 2, \ldots, q_i\}}$ be a sequence of non-negative integers. We define an increasing sequence of finite words over the alphabet $\{0, 1\}$. We start with $\omega_0=1$, and let $\omega_{n}$ be a concatenation of $q_n$ copies of $\omega_{n-1}$, where the $j$-th copy is followed immediately by $a_{n, j}$ zeroes.
Denote by $h_n$ the length of $\omega_n$. 

Let $X\subset \{0, 1\}^{\mathbb{Z}}$ be the subshift consisting of all bi-infinite words, whose all subwords are subwords of one of the $\omega_n$. Since each $\omega_n$ is the start of $\omega_{n+1}$, there is a word $x\in X$ such that $x[0, h_n]=\omega_n$; this point $x$ is transitive in $X$. From now on, by $x$ we will always mean this specific point of the system. 

The point $x$ is generic for a probability measure $\mu$, which can be described explicitly: for any word $w$ of length $\ell$ we have
\[
\mu([w])=\lim_{n\to\infty}\frac{1}{h_n}\#\{i\in [0, h_n): x[i, i+\ell]=w\}.
\]
Denote the shift by $\sigma$. We then call $(X, \sigma, \mu)$ a rank-one map.

We recall the following results:
\begin{prop} \cite{Ornstein1972} There exist infinite mixing rank-one maps.
\end{prop}
\begin{prop}\cite{Kalikow1984}
    In the class of rank-one maps, mixing implies mixing of all orders.
\end{prop}

\subsection{Proof of Theorem \ref{mainA}}
The proof of the theorem proceeds in two steps. First, we lower bound the density of occurrences of the words $\omega_n$ in $x$. We then use this to prove the key intermediate result, Theorem \ref{thm:lem}, which allows us to apply Proposition \ref{prop:IC} and complete the proof.

Fix a nonnegative integer $k$. By the definition of the sequence $(\omega_n)$, we recursively define for $n\geqslant k$ a splitting of each $\omega_n$ into blocks $\omega_k$ and sequences of zeroes. Such a splitting extends to a splitting of $x$. Let $\tilde W_k\subset\mathbb{Z}$ be the set of all starting positions of the blocks $\omega_k$ of $x$ under this splitting. 

\begin{prop}\label{prop:density}
If the measure $\mu$ is non-atomic, we have
\[
\liminf_{n\to\infty}h_n\, \ud(\tilde W_n)\geq 1.
\]
Moreover for every $k\in\Z^+$ and $w\in X[0,k]$, the measure of the cylinder $\mu([w])$ {is positive}.
\end{prop}
\begin{proof}

If $\mu$ is non-atomic, $\mu([1])>0$ (otherwise $\mu=\delta_{\{0\}^\Z}$). We have
\[
\mu([1])=\lim_{n\to\infty}\frac{1}{h_n}\#\{i\in [0, h_n): x[i]=1\}=\lim_{n\to\infty}\frac{1}{h_n}\prod_{k=1}^{n}q_k
=\lim_{n\to\infty}\prod_{k=1}^{n}\left(1-\frac{\sum_{j=1}^{q_k}a_{k,j}}{h_{k}}\right),
\]
where we have used the identity $h_k=h_{k-1}q_{k}+\sum_{j=1}^{q_k}a_{k,j}$ for every $k$.
Hence, the limit 
\[
\lim_{n\to\infty}\prod_{m=1}^{n}\left(1-\frac{\sum_{j=1}^{q_m}a_{m,j}}{h_{m}}\right)>0,
\] 
which implies
\[
\lim_{k\to\infty}\lim_{n\to\infty}\prod_{m=k+1}^{n}\left(1-\frac{\sum_{j=1}^{q_m}a_{m,j}}{h_{m}}\right)=1.
\]
Therefore, for every $k$, by the definition of $\tilde W_k$ we have
\[
h_k\,\ud(\tilde W_k)\geq\lim_{n\to\infty}\frac{h_k}{h_n}\#(\tilde W_k\cap[0,h_n))=\lim_{n\to\infty}\frac{h_k}{h_n}\prod_{m=k+1}^n q_m=\lim_{n\to\infty}\prod_{m=k+1}^n \left(1-\frac{\sum_{j=1}^{q_m}a_{m,j}}{h_{m}}\right)\xrightarrow{k\to\infty}1.
\]

Now, given $k\in\Z^+$ and $w\in X[0,k]$, there exists $m$ such that $w$ is a subword of $\omega_m$, so 
\[
\mu([w])\geq \mu([\omega_m])=\lim_{n\to\infty}\frac{1}{h_n}\#\{i\in [0, h_n): x[i, i+\ell]=\omega_m\}\geq \lim_{n\to\infty}\frac{1}{h_n}\#(\tilde W_m\cap[0,h_n))>0,
\]
by the computations above.
\end{proof}

The main part of the proof is the following theorem.

\begin{thm}\label{thm:lem}
Fix a mixing rank-one subshift $X$ and an integer $k$. There exists positive integers $M$ and $n_0$ such that for any subset $A\subset X$ and any $n>n_0$ one of the following holds:
\begin{itemize}
\item at least $0.1h_n$ integers $i\in [0, 3h_n)$ satisfy $A[i, i+k]=X[0, k]$;
\item we have $\# A[h_n, 2h_n]\leqslant M$.
\end{itemize}
\end{thm}
\begin{proof}

The system $X$ is mixing of all orders, so we can use Lemma \ref{lemma:countdisc} for the partition of $X$ into cylinders of length $k$, to obtain an integer $K$. Let $M=2^{200K}$. 

Assume that the theorem does not hold. Then we obtain an infinite set $\Ex\subset\mathbb{N}$ and sets $A_n\subset X$ for each $n\in \Ex$ such that neither of the conditions in the statement hold for $A_n$.

For any $y\in X$ and $n\in \N$, define the set
\[
W_n(y):=\{i\in\mathbb{Z}: y[i, i+h_n]=\omega_n\}.
\]
For each $n$, $\tilde W_n\subset W_n(x)$,
but they might not be equal. Note the following simple consequence of the definition of sets $W_n$:
\begin{claim}\label{claim1}
If two elements $y, z\in X$ satisfy $W_n(y)\cap [0, 2h_n)=W_n(z)\cap [0, 2h_n)$, then $y[h_n, 2h_n]=z[h_n, 2h_n]$.
\end{claim}
\begin{proof}
For any $y\in X$ and any $i\in\mathbb{Z}$, if $y[i]=1$ then $W_n(y)$ contains an element $j\in [i-h_n+1, i]$ (in other words, $y[i]$ is a part of a subword $\omega_n$ of $y$). Indeed, the set of all $y$ satisfying this property is closed, and it contains the orbit of $x$, since $x$ admits a splitting into blocks $\omega_n$ and zeroes. Since $x$ is transitive in $X$, all elements of $X$ have the above property.

But now, for any $y\in X$ and $i\in [h_n, 2h_n)$, the symbol $y[i]$ is uniquely determined by $W_n(y)\cap [0, 2h_n)$. Indeed, if there exists a $j\in W_n(y)\cap [i-h_n+1, i]$, then $y[i]=\omega_n[i-j]$, and otherwise $y[i]=0$.
\end{proof}
Therefore consider the set
\[
W_n=[0, 2h_n)\cap  \bigcup_{y\in A_n}W_n(y).
\]
\begin{claim}
For $n\in \Ex$ we have $\#W_n> 200K$.
\end{claim}
\begin{proof}
 Otherwise the family
\[
\{W_n(y)\cap [0, 2h_n): y\in A_n\}
\]
has at most $2^{200K}=M$ elements, since each element of this family is a subset of $W_n$. But by Claim \ref{claim1}, each word in $A_n[h_n, 2h_n]$ corresponds to a distinct member of this family. Therefore the second condition from the theorem statement holds in this case, and we get a contradiction.
\end{proof}

Now suppose that $\#W_n>200K$. Then there exists an integer $t\in[0, 2h_n)$ such that $\#W_n\cap [t, t+0.01h_n)\geqslant K$. But then the set $A_n[t+0.01h_n, t+h_n]$ contains $K$ subwords of $\omega_n$ starting at different positions of $\omega_n$. Since $\omega_n=x[0, h_n]$, this means precisely that there exist integers
\[0\leqslant s_{n, 1}<s_{n, 2}<\ldots <s_{n, K}<0.01h_n\] such that \[x[s_{n, i}, s_{n, i}+0.99h_n]\in A_n[t+0.01h_n, t+h_n]\] for every $i$. 

Consider therefore the tuple $(\sigma^{s_{n, i}}x)_{i\in \{1, 2, \ldots, K\}}\in X^{\times K}$, which we will denote by $\alpha_n$. Since $x$ is a generic point for $\mu$, this tuple is a generic point for the properly off-diagonal $K$-fold joining of $\sigma$ of the form
\[
\mu_n=(\sigma^{s_{n, 1}}\times \ldots \times \sigma^{s_{n, K}})\mu
\]

\begin{claim}
    Possibly after replacing $\Ex$ by an infinite subset, we have that 
    \[
    \lim_{\substack{n\to\infty\\n\in \Ex}}\mu_n=\tilde{\mu}
    \]
    in the weak$^*$ sense, for some non-degenerate trivial $K$-fold joining $\tilde{\mu}$ of $\sigma$. 
\end{claim}
\begin{proof}
    Take a subset of $\Ex$ such that along this subset, for each $i$ the sequence $(s_{n, i+1}-s_{n, i})$ is either constant or strictly increasing. Then the index set $\{1, 2, \ldots, K\}$ splits into intervals where the differences are constant, so the marginals on these intervals converge to properly off-diagonal measures. Between these intervals the differences grow to infinity, so by mixing of all orders the interval marginals become independent, and $\tilde{\mu}$ is the product of interval marginals: so, a non-degenerate trivial self-joining.
\end{proof}
Since the system $X$ is weakly mixing, each $\mu_n$ and $\tilde{\mu}$ are ergodic.

We now want to relate the measures $\mu_n$ back to the subwords of length $k$ appearing in $A_n[t+0.01h_n, t+h_n]$. Therefore, consider the measures 
\[
\nu_n=\frac{1}{0.99h_n}\sum\limits_{i<0.99h_n}\delta_{\sigma^i \alpha_n}.
\]
\begin{claim}
Possibly after replacing $\Ex$ by an infinite subset, we have
\[
    \lim_{\substack{n\to\infty\\n\in \Ex}}\nu_n=\nu,
    \]
    in the weak$^*$ sense, where $\nu$ is a probability measure on $X^{\times N}$ invariant under $\sigma\times\ldots\times\sigma$.
\end{claim}
\begin{proof}
    For each $n$, $\nu_n$ is a probability measure on $X^{\times N}$, and the space of probability measures on $X^{\times N}$ is weakly$^*$ compact. Therefore $\{\nu_n\}_{n\in\Ex}$ has a convergent subsequence to a probability measure $\nu$. By definition of $\nu_n$, 
    \[(\sigma\times\ldots\times\sigma)_*\nu_n=\nu_n+\frac{1}{0.99h_n}\left(\delta_{\sigma^{\lceil0.99h_n\rceil} \alpha_n}-\delta_{ \alpha_n}\right),\] 
    so the difference goes to 0 as $n\to\infty$ and hence the limit $\nu$ is invariant under $\sigma\times\ldots\times\sigma$.
\end{proof}
The measure $\nu$ is in fact equal to $\tilde\mu$. We show this in two steps.
\begin{claim}
    For any measurable set $B\subset X^{\times K}$ we have $\tilde\mu(B)\geqslant 0.98\nu(B)$.
\end{claim}
\begin{proof}
     It is enough to show this for sets of the form $B=B_1\times\ldots\times B_K$, where the sets $B_i$ are cylinders of common length, say $\ell$. 

    Recall the splitting of $x$ into blocks $\omega_n$ starting at elements of $\tilde W_n$. If a number $i\in [\ell, 0.99h_n-\ell]$ satisfies $\sigma^i \alpha_n\in B$, then $\sigma^{m+i}\alpha_n\in B$ for any $m\in \tilde W_n$. Therefore
    \[
    \#\{i\in [0, N): \sigma^i\alpha_n\in B\}\geqslant \#\{i\in [\ell, 0.99h_n-\ell]: \sigma^i\alpha_n\in B\}\cdot|\tilde W_n\cap [0, N-h_n]|.
    \]
    Dividing by $N$ and taking it to infinity we obtain
    \[
    \mu_n(B)\geqslant \Big(0.99h_n\cdot\nu_n(B)-2\ell\Big)\cdot \ud(\tilde W_n).
    \]
    Taking $n$ to infinity along the infinite set $\Ex$, and using Proposition \ref{prop:density}, we obtain
    \[
    \tilde\mu(B)\geqslant 0.99\cdot 0.99\nu(B)\geqslant 0.98\nu(B).
    \]
    
\end{proof}
\begin{claim}
    We have $\nu=\tilde\mu$.
\end{claim}
\begin{proof}
    Let $\nu'=\tilde\mu-0.98\nu$. A priori $\nu'$ is a signed measure on $X^{\times K}$, but by the previous claim it is in fact a positive measure. It is also clearly $\sigma\times\ldots\times \sigma$-invariant. We already know that $\tilde{\mu}$ is ergodic, and the equation \[
    \tilde\mu=0.98\nu+\nu'
    \]
    is a presentation of $\tilde\mu$ as a linear combination of two other $\sigma$-invariant measures, with $\nu(X^{\times K})=1=\tilde{\mu}(X^{\times K})$, so we must have $\tilde\mu=\nu={50}\nu'$. 
\end{proof}

Now consider the set $B\subset X^{\times K}$ of all tuples which visit every element of the partition $X$ into the cylinders given by $X[0,k]$, which have positive measure by Proposition \ref{prop:density}. Since $B$ is clopen, and by Lemma \ref{lemma:countdisc}, we have
\[
\lim_{\substack{n\to\infty\\n\in \Ex}}\nu_n(B)=\nu(B)=\tilde{\mu}(B)>\frac{1}{2}.
\]
Hence, for large enough $n\in \Ex$ we have 
\[
\#\{i\in [t+0.01h_n, t+h_n]: A_n[i, i+k]=X[0, k]\}\geqslant 0.99h_n\cdot \nu_n(B)-2k>0.1h_n.
\]
But then for these $n$ the first condition in the theorem statement is satisfied. This contradicts the definition of $\Ex$, and ends the proof.

\end{proof}

Now we move on to proving the main theorem of the section.

\begin{proof}[Proof of Theorem \ref{mainA}]
Assume that the system $X$ is not IC-rigid, and let $A\subset X$ be the set provided by Proposition \ref{prop:IC}. Then we obtain a number $k$ such that the set
\[
E_1=\{n\in\mathbb{Z}^+: A[n, n+k]=X[0, k]\}
\]
has upper density $<0.01$. Let us now use Theorem \ref{thm:lem}, obtaining the integer $M$. By the choice of $A$, there is a length $\ell$ such that the set
\[
E_2=\{n\in\mathbb{Z}^+: \#A_n[n, n+\ell]\leqslant M\}
\]
has upper density $<0.01$.

Now let us pick an $n>n_0$ for which $h_n>2\ell$.

\begin{claim*}
    For any $i\in\mathbb{Z}^+$, the set $E_1\cup E_2$ has at least $0.1h_n$ elements in the interval $[i, i+3h_n)$. 
\end{claim*}
\begin{proof}
Use the conclusion of Theorem \ref{thm:lem} for the set $\sigma^{-i}A$. We obtain that either $E_1\cap [i, i+3h_n)$ has at least $0.1h_n$ elements, or the set $A[i+h_n, i+2h_n]$ has at most $M$ elements. But in the second case we obtain at least $h_n-\ell\geqslant h_n/2$ elements in $E_2\cap [i, i+3h_n)$.
\end{proof}

By dividing $\mathbb{Z}^+$ into intervals of length $3h_n$, we see that $E_1\cup E_2$ has upper density at least $\frac{1}{30}$, so $E_1$ and $E_2$ cannot both have upper density $<0.01$. This ends the proof.

\end{proof}

\section{Weakly asymptotic pairs and IC-rigidity}\label{sec:WP}

We begin by defining the notion of weakly asymptotic pairs.

\begin{defn}[Weakly asymptotic pairs] Let $X$ be a metric space with a distance $\dist$. We say that $x,y\in X$, $x\neq y$, are a $G$-weakly asymptotic pair if there exists a set $D\subset G$ with $\ud (D)>0$ such that $\lim\dist(gx,gy)=0$ for $g\to\infty$ along the set $D$.
\end{defn}

The goal of this section is to prove Theorem \ref{mainC}, with the ultimate goal of applying them to horocycle flows in Section \ref{sec:hf}. We actually prove the following modified version:

\begin{thm} \label{thm:generic} Let $X$ be a compact metric space and let $\mu$ be a probability measure on $X$ invariant under the action of a group $G$. Suppose $(X,G,\mu)$ is minimal, uniquely ergodic, mixing of all orders, and has minimal self-joinings of all orders. Moreover, assume that every point $(x_1,\cdots,x_N)\in X^{\times N}$ is generic for an $(X^{\times N},\{g^{\times N}\}_{g\in G})$-ergodic measure, and that for every $g$ in the center of $G$ $\{(x,y)\in X\times X$ generic for $\prescript{2}{}{\Delta}^{(e,g)}\}=\{(x,gx):x\in X\}$. Then $(X,G)$ is IC-rigid.
\end{thm}

We prove Theorem \ref{thm:generic} in Section \ref{sec:proofWP}. The fact that it is equivalent to Theorem \ref{mainC} is an immediate consequence of the following lemma. Recall that $Z(G)$ denotes the center of $G$.
\begin{lem}\label{lemma:WPgen} Suppose $(X,G,\mu)$ is uniquely ergodic and has 2-fold minimal self-joinings. Moreover, assume that every pair $(x,y)\in X\times X$ is generic for an $(X\times X,\{g\times g\}_{g\in G})$-ergodic measure. The following are equivalent:
\begin{itemize}
    \item[i)] The system has no weakly asymptotic pairs.
    \item[ii)] For every $g\in Z(G)$, $\{(x,y)\in X\times X$ generic for $\prescript{2}{}{\Delta}^{(e,g)}\}=\{(x,gx):x\in X\}$.
\end{itemize}
\end{lem}

    \begin{proof}

    {i)$\implies$ii)}. We show that for every $(x_0,y_0)\in X\times X$ generic for $\prescript{2}{}{\Delta}^{(e,g)}$, either $y_0=gx_0$ or $(gx_0,y_0)$ is a weakly asymptotic pair. 
    
    For $l\geq 1$, let $D_l:=\left\{ h\in G:\dist(hg x_0,hy_0)<{l}^{-1}\right\}$. Since $D_{l+1}\subset D_l$, $\ud (D_l)$ is decreasing, so
        \[\begin{aligned}
       1=\prescript{2}{}{\Delta}^{(e,g)} \left(\{(x,gx):x\in X\}\right) =\lim_{l\to\infty}\prescript{2}{}{\Delta}^{(e,g)}(\{ (x,y)\in X\times X:\dist(g x,y)< {l}^{-1}\})\leq \inf_{l\geq 1} \ud (D_l).
        \end{aligned}
        \]
        For every $l\geq 1$, let $K_l\geq 1$ be such that ${m(F_{k})}^{-1}m(D_l\cap F_{k})>1/2$ for every $k\geq K_l$ and set \[D:=\bigcup_{l\geq 1}\left(D_l\cap F_{{K_{l+1}}}\right).\]
        Then $\ud(D)>1/2$ and $\lim\dist(hgx_0,hy_0)=0$ for $h\to\infty$ along the set $D$.

        {ii)$\implies$i)}. We show it by contradiction. Suppose $(x_0,y_0)\in X\times X$, $x_0\neq y_0$, is a weakly asymptotic pair and let $\nu$ be the ergodic measure for which $(x_0,y_0)$ is generic. 
        
        For $l\geq 1$, let $ \tilde D_l:=\left\{ h\in G:\dist(hx_0,hy_0)\leq{l}^{-1}\right\}$. Since $(x_0,y_0)$ is weakly asymptotic, there exists a set $\tilde D\subset G$ with $\ud(\tilde D)>0$ such that $\lim_{h\to\infty}\dist(hx_0,hy_0)=0$ as $h\to\infty$ along $\tilde D$. In particular, $\ud(\tilde D_l)\geq\ud(\tilde D)$ for every $l\geq 1$, so
        \[
        \nu\left(\{(x,x):x\in X\}\right) = \lim_{l\to\infty} \nu\left(\{(x,y)\in X\times X:\dist(x,y)\leq l^{-1}\}\right)\geq \inf_{l\geq 1}\ud(\tilde D_l)>0.
        \]
        By assumption, either $\nu=\mu\times \mu$ or $\nu =\prescript{2}{}{\Delta}^{(e,g)} $ for some $g\in Z(G)$, so $g=e$. Thus, we found a pair $(x_0,y_0)$ generic for $\prescript{2}{}{\Delta}^{(e,e)}$ with $x_0\neq y_0$.
\end{proof}

\subsection{Proof of Theorems \ref{mainC} and \ref{thm:generic}}\label{sec:proofWP}
By Lemma \ref{lemma:WPgen} it is enough to prove one of the two results.
\begin{proof}[Proof of Theorem \ref{thm:generic}]
   
    Suppose not and let $A$ be the compact set given by Proposition \ref{prop:IC}. We know from condition \textit{b)} that $A\not\in\K_{G}(X)$. Moreover, note that $g^{-1}h\in Z(G)$ is an equivalence relation on $G$. Therefore, $A$ contains either a sequence of the form $\{g_nx\}_{n\geq 1}$ for some $x\in X$ where $g_1=e$ and $g_n\to\infty$ in $Z(G)$ (case 1), a sequence of the form $\{g_nx\}_{n\geq 1}$ for some $x\in X$, where $g_1=e$ and $g_n\in G \setminus Z(G)$ for $n\geq 2$ are pair-wise distinct (case 2), or an infinite set of points with pairwise disjoint orbits (case 3).

    Fix $\eps>0$ and let $\cP$ be a finite cover of $X$ with open sets of diameter $<\eps$. The system is mixing of all orders, so let $N$ and $L$ be the integers given by applying Lemmas \ref{lemma:countprod} and \ref{lemma:countoffd} to $\cP$. 

    In case 1, there exists $\{g_i\}_{i=1}^{N}\subset Z(G)$ with $g_i^{-1}g_j\not\in F_L$ for every $i\neq j$, and $x\in X$ such that $\{g_ix\}_{i=1}^{N}\subset A$. Then, letting $\textbf{g}=(g_{1},\ldots,g_{N})$, the point $(g_{1}x,\ldots,g_{N}x)$ is generic for $\prescript{N}{}{\Delta}^{\textbf{g}}$, so by our choice of $N$, $L$, and $\cP$, for $D_\eps:=\{n\in\Z^+: T^n A \text{ is }\varepsilon\text{-dense in }X\}$, the following holds:
\[\begin{aligned}
    \ud(D_\eps)&\geq\limsup_{n\to\infty} \frac{1}{m(F_n)}\#\{h\in F_n: \{hg_{1}x,\ldots,hg_{N}x\}\text{ intersects every }P\in\cP\}\\&=
    \prescript{N}{}{\Delta}^{\textbf{g}}\biggl(\Big\{(y_1, \ldots, y_N)\in X^{\times K}: \{y_1, \ldots, y_N\} \text{ intersects every $P\in\mathcal{P}$}\Big\}\biggl)\,>1/2.
    \end{aligned}
    \]

    For cases 2 and 3, there exists a set $\{x_1,\ldots,x_N\}\subset A$ such that no pair is related by an element in $Z(G)$: there does not exist $ 1\leq r<s\leq N$ and $g\in Z(G)$ such that $gx_r=x_s$.
    \begin{claim*} If $p=(x_1,\ldots,x_N)\in X^{\times N}$ does not have any two coordinates related by an element in the center of $G$, then $p$ is generic for $\mu^{\times N}$.
    \end{claim*}
    \begin{proof} 
        Suppose not, then by assumption $p$ is generic for a different $(X^{\times N},\{g^{\times N}\}_{g\in G})$-ergodic measure $\nu$. Moreover, $(X,G,\mu)$ is uniquely ergodic and has minimal self-joinings, so $\nu\in\Joi_N(G,\mu)$ and it is trivial. Therefore, there exist $1\leq r< s\leq N$ such that $(x_r,x_s)$ is generic for an off-diagonal joining $\prescript{2}{}{\Delta}^{(e,g)}$ in $X\times X$, for some $g$ in the center of $G$. But this implies $gx_r= x_s$, which contradicts our hypothesis.
    \end{proof}
    Therefore, by our choice of $N$ and $\cP$,
    \[\begin{aligned}
    \ud(D_\eps)&\geq\limsup_{n\to\infty} \frac{1}{m(F_n)}\#\{g\in F_n: \{gx_1,\ldots,gx_N\}\text{ intersects every }P\in\cP\}\\
    &=\mu^{\times N}\biggl(\Big\{(y_1, \ldots, y_N)\in X^{\times K}: \{y_1, \ldots, y_N\} \text{ intersects every $P\in\mathcal{P}$}\Big\}\biggl)\,>1/2.
    \end{aligned}
    \]
    
    Since this holds for arbitrary $\eps>0$, it contradicts \textit{a)} in Proposition \ref{prop:IC}.

\end{proof}

\section{Actions of countable abelian groups}\label{sec:WP2}

The assumption of Theorem \ref{thm:generic} that every point $(x_1,\cdots,x_N)$ in $X^{\times N}$ is generic for an ergodic measure of the product system $(X^{\times N},\{g^{\times N}\}_{g\in G})$ appears rather strong. However, under the additional assumption that $G$ is countable and abelian, this property follows from the remaining hypotheses. 

\begin{lem}\label{lemma:genWP} Let $G$ be countable and abelian. Suppose the system $(X,G,\mu)$ is uniquely ergodic, has minimal self-joinings of all orders, and has no weakly asymptotic pairs. Then every point $(x_1,\cdots,x_N)\in X^{\times N}$ is generic for an $(X^{\times N},\{g^{\times N}\}_{g\in G})$-ergodic measure.
\end{lem}

By Theorem \ref{mainC}, Lemma \ref{lemma:genWP} immediately yields Theorem \ref{mainD}. 

\subsection{Proof of Theorem \ref{mainD}}\label{sec:lemmaWP}

It is enough to prove Lemma \ref{lemma:genWP}, which in turn is a corollary of the following lemma.
\begin{lem} \label{lemma:WP}Let $G$ be countable (not necessarily abelian). Suppose the system $(X,G,\mu)$ is uniquely ergodic, has minimal self-joinings of all orders, and has no weakly asymptotic pairs. Then every point $(x_1,\cdots,x_N)\in X^{\times N}$ with no pair of them lying in the same orbit is generic for $\mu^{\times N}$.
\end{lem}
\begin{proof}
    Suppose, for the sake of contradiction, that there exists a point $p\in X^{\times N}$ whose points lie in distinct $G$-orbits which is not generic for $\mu^{\times N}$, and define
    \[
    \nu_n=\frac{1}{m(F_n)}\int_{F_n}\delta_{g^{\times N}p}\,dm(g).
    \]
    Then there exists a $\{g^{\times N}\}_{g\in G}$-invariant Borel probability measure $\nu\neq\mu^{\times N}$ such that 
    $\nu_n\to \nu$ in the weak$^*$ sense, along a subsequence $\{n_k\}_{k\in\N}$.

    Since $(X,G,\mu)$ is uniquely ergodic and has minimal self-joinings of all orders, the set of ergodic measures for $(X^{\times N},\{g^{\times N}\}_{g\in G})$ is exactly $\Joi_N(G,\mu)$. In particular, it is countable, so by the ergodic decomposition theorem, there exist an element $\mu^{\times N}\neq\tilde \nu\in \Joi_N(G,\mu)$ and a positive number $\alpha\in (0,1]$, such that $\nu\geq\alpha\tilde\nu$ on $X^{\times N}$.

    For each $1\leq i < j \leq N$, let \[\pi_{i,j}:X^{\times N}\to X\times X,\qquad\pi_{i,j}(x_1,\cdots,x_N)=(x_i,x_j),\] denote the coordinate projection.
    Since $\mu^{\times N}\neq\tilde \nu\in \Joi_N(G,\mu)$, there exist two coordinates $1\leq i < j \leq N$ and $g$ in the center of $G$ such that 
    \[
    (\pi_{i,j})_*\tilde\nu=\prescript{2}{}{\Delta}^{(e,g)},
    \] 
    where $e$ is the identity in $G$. That is, the projection of $\tilde\nu$ onto the $(i,j)$-coordinates is off-diagonal.

    \begin{claim*} Let $(x_0,y_0):=\pi_{i,j}p$. Then 
        $(g x_0,y_0)$ is a weakly asymptotic pair.
    \end{claim*}
    
    \begin{proof}
        The proof is analogous to that of Lemma \ref{lemma:WPgen}; we include it for completeness. 
        Note 
        \[
         (\pi_{i,j})_*\nu(\{(x,g x):x\in X\})\geq \alpha\,(\pi_{i,j})_*\tilde\nu(\{(x,g x):x\in X\})=\alpha>0.
        \]
    For each $l\in\Z^+$, let 
    \[
    D_l:=\left\{ h\in G:\dist(hg x_0,hy_0)<{l}^{-1}\right\}, \quad B_l:=\left\{ (x,y)\in X\times X:\dist(g x,y)< {l}^{-1}\right\}.
    \]
    Then, since $\nu_n$ converges along $\{n_k\}$, $B_l$ are open sets, and $D_{l+1}\subset D_l$ for every $l$,
        \[\begin{aligned}
        \alpha&\leq(\pi_{i,j})_*\nu(\{(x,gx):x\in X\})=\lim_{l\to\infty}(\pi_{i,j})_*\nu(B_l)\leq\lim_{l\to\infty}\liminf_{k\to\infty}\frac{1}{m(F_{n_k})}\int_{F_{n_k}}\delta_{(hx_0,hy_0)}(B_l)\,dm(h)\\
        &=\lim_{l\to\infty}\liminf_{k\to\infty} \frac{m\left(D_l\cap F_{n_k}\right)}{m(F_{n_k})}=\inf_{l\geq 1}\liminf_{k\to\infty}\frac{m\left(D_l\cap F_{n_k}\right)}{m(F_{n_k})}.
        \end{aligned}
        \]
        For every $l\geq 1$, let $K_l\geq 1$ be such that ${m(F_{n_k})}^{-1}m(D_l\cap F_{n_k})\geq\frac{\alpha}{2}$ for every $k\geq K_l$ and set \[D:=\bigcup_{l\geq 1}\left(D_l\cap F_{n_{K_{l+1}}}\right).\]
        Then $\ud(D)\geq \frac{\alpha}{2}>0$ and $\lim\dist(hgx_0,hy_0)=0$ for $h\to\infty$ along the set $D$.
    \end{proof}
    Since we assumed that the system does not have weakly asymptotic pairs, this implies $y_0=gx_0$, which contradicts the hypothesis that $p$ does not have any pair of coordinates lying in the same $G$-orbit.

\end{proof}

\begin{proof}[Proof of Lemma \ref{lemma:genWP}]
    Let $p=(x_1,\cdots,x_N)\in X^{\times N}$ and consider the equivalence relation $x\sim y$ if $x$ and $y$ lie in the same orbit. Let the index set ${i_1,\ldots,i_r}$ represent the equivalence classes in $p$, so that $(x_{i_1},\ldots,x_{i_r})$ lie in distinct orbits. Then, by Lemma \ref{lemma:WP}, $(x_{i_1},\ldots,x_{i_r})$ is generic for $\mu^{\times r}$.

    For every $1\leq j\leq N$ there exists a unique $1\leq \sigma(j)\leq r$ and $g(j)\in G$ such that $x_j=g(j)x_{i_{\sigma(j)}}$, so define the function
    \[
    \phi:X^{\times r}\to X^{\times N},\qquad \phi(y_1,\ldots,y_r)=\left(g(1)y_{\sigma(1)},\ldots,g(N)y_{\sigma(N)}\right).
    \]
    Since $G$ is abelian, $\phi\circ g^{\times r}=g^{\times N}\circ \phi$ for every $g\in G$, so the point $p$ is generic for the measure $\phi_*\,\mu^{\times r}$, which is a trivial $N$-fold joining, hence ergodic.
\end{proof}

\subsection{Topological minimal self-joinings} We now proceed to the proof of Theorem \ref{main:Z}. First, we recall the definition of topological minimal self-joinings (see \cite{delJunco1987,King_1990}).

\begin{defn}[Topological minimal self-joinings]
    For any group $G$, say the system $(X, G)$ has \textit{$N$-fold topological minimal self-joinings} (TMSJ), for $N\geq 2$, if for every point $(x_1,\cdots,x_N)\in X^{\times N}$ with no pair of them lying in the same orbit, its orbit under $T^{\times N}$ is dense in $X^{\times N}$. 

    If a system has $N$-fold TMSJ for every $N\geq 2$, we say it has TMSJ of all orders.
\end{defn}

As it turns out, the hypotheses of Theorem \ref{mainC} imply topological minimal self-joinings of all orders.
\begin{prop}
    \label{prop:topmsj} If a system satisfies the hypothesis of Theorem \ref{mainC}, then it has TMSJ all orders.
\end{prop}
\begin{proof}
Fix $N\geq 2$. Recall from the Claim in the proof of Theorem \ref{thm:generic} in Section \ref{sec:proofWP} that if $p=(x_1,\ldots,x_N)\in X^{\times N}$ does not have any two coordinates related by an element of the center of $G$, then $p$ is generic for $\mu^{\times N}$. Since the system is minimal and uniquely ergodic, this implies that its orbit under $T^{\times N}$ is dense in $X^{\times N}$.
\end{proof}
By Lemma \ref{lemma:genWP}, we have the following immediate corollary:
\begin{cor}\label{cor:topmsj}
    If a system satisfies the hypothesis of Theorem \ref{mainD}, then it has TMSJ of all orders.
\end{cor}

However, having topological minimal self-joinings of all orders is a very strong property -- in particular, by a recent result of Bitar, Donoso and Petite, for actions of finitely generated abelian groups it holds only in trivial systems.
\begin{thm}\cite[Theorem D]{BitarDonosoPetite2026}\label{prop:king} Let $G$ be a finitely generated abelian group of rank $d$. No infinite system $(X,G)$ has $(2d+2)$-fold topological minimal self-joinings. 
\end{thm}

\begin{proof}[Proof of Theorem \ref{main:Z}]
Theorem \ref{prop:king} and Corollary \ref{cor:topmsj} imply that the conditions of Theorem \ref{mainD} are unachievable in this setting.
\end{proof}

\section{Some horocycle flows are IC-rigid}\label{sec:hf}

The goal of this section is to prove Theorem \ref{mainB}, hence providing examples of IC-rigid $\R$-actions. We do so by showing that the system satisfies the hypotheses of Theorem \ref{thm:generic}. First, we recall some classical results.

\begin{prop}\cite{Hedlund1936,Furstenberg1973} \label{prop:hf1}
    Let $\Gamma$ be a cocompact lattice in $\PSL_2(\mathbb{R})$, let $\mu$ denote the Haar measure on $X=\PSL_2(\mathbb{R})/\Gamma$, and let $h_t$ denote the horocycle flow in $X$. Then $(X,h_t,\mu)$ is minimal and uniquely ergodic.
\end{prop}

\begin{prop}\cite{Parasyuk1953,Marcus1978} \label{prop:hf2}
    Let $\Gamma$ be a lattice in $\PSL_2(\mathbb{R})$, let $\mu$ denote the Haar measure on $X=\PSL_2(\mathbb{R})/\Gamma$, and let $h_t$ denote the horocycle flow in $X$. Then $(X,h_t,\mu)$ is mixing of all orders.
\end{prop}

\begin{prop}\cite{Ratner1983} \label{prop:hf3}
    Let $\Gamma$ be a maximal nonarithmetic lattice in $\PSL_2(\mathbb{R})$, let $\mu$ denote the Haar measure on $X=\PSL_2(\mathbb{R})/\Gamma$, and let $h_t$ denote the horocycle flow in $X$. Then $(X,h_t,\mu)$ has minimal self-joinings of all orders.
\end{prop}

\begin{prop}\cite{Ratner1991} \label{prop:gen}
     Let $G$ be a connected Lie group, $U$ a one-parameter unipotent subgroup of $G$, $\Gamma$ a lattice in $G$, and $x\in G/\Gamma$. Then $x$ is generic for the unique $U$-invariant ergodic probability measure supported on the homogeneous orbit closure $\overline{Ux}$.
\end{prop}

In particular, we obtain the following two corollaries.
\begin{prop} \label{prop:hf4}
    Let $\tilde \Gamma$ be a lattice in $\PSL_2(\mathbb{R})$, let $\mu$ denote the Haar measure on $X=\PSL_2(\mathbb{R})/\Gamma$, and let $h_t$ denote the horocycle flow in $X$. Then every point $(x_1,\cdots,x_N)\in X^{\times N}$ is generic for an $(X^{\times N},\{h_t^{\times N}\}_{t\in \R})$-ergodic measure.
\end{prop}
\begin{proof}
    For every $N\in\Z^+$, apply Proposition \ref{prop:gen} to $G=\PSL_2(\mathbb{R})^{\times N}$, $\Gamma=\tilde\Gamma^{\times N}$, and $U=(h_t^{\times N})_{t\in\R}$.
\end{proof}
\begin{prop} \label{prop:hf5}
Let $\Gamma$ be a lattice in $\PSL_2(\mathbb{R})$, let $\mu$ denote the Haar measure on $X=\PSL_2(\mathbb{R})/\Gamma$, and let $h_t$ denote the horocycle flow in $X$. Then for every $t\in\R$, $(x,y)\in X\times X$ is generic for $(Id\times h_t)_*\mu$ if and only if $y=h_tx$.
\end{prop}
\begin{proof}
Clearly $\{(x,h_tx):x\in X\}\subset\{(x,y)\in X\times X$ generic for $(Id\times h_t)_*\mu\}$, so it is enough to show $\{(x,y)\in X\times X$ generic for $(Id\times h_t)_*\mu\}\subset\{(x,h_tx):x\in X\}$. 

    Apply Proposition \ref{prop:gen} with $G=\PSL_2(\mathbb{R})\times \PSL_2(\mathbb{R})$, $\Gamma\times \Gamma$, and $U=(h_s\times h_s)_{s\in\R}$, and for every $(x,y)\in X\times X$ let $\nu_{x,y}$ be the unique $U$-invariant ergodic probability measure for which $(x,y)$ is generic. Then, if $\nu_{x,y}=(Id\times h_t)_*\mu$,
    \[(x,y)\in\text{Supp}(\nu_{x,y})=\text{Supp}((Id\times h_t)_*\mu)=\{(x,h_tx):x\in X\}.\]
\end{proof}

We are now ready to prove Theorem \ref{mainB}.
\begin{proof}[Proof of Theorem \ref{mainB}]
    By Propositions \ref{prop:hf1}, \ref{prop:hf2}, \ref{prop:hf3}, \ref{prop:hf4}, and \ref{prop:hf5}, the horocycle flow satisfies the hypotheses of Theorem \ref{thm:generic}. Thus, it is IC-rigid.
\end{proof}

Note that, by Proposition \ref{prop:topmsj}, such horocycle flows provide an example of $\R$-actions with topological minimal self-joinings of all orders.

\section{IC-rigid actions of direct sums of finite groups}\label{sec:dani}

\subsection{Properties of $(C, F)$-actions of direct sums of finite groups}

In this subsection we recall the construction and properties of $(C, F)$-actions from \cite{Danilenko_2006}.

Let $(G_i)$ be a sequence of finite groups with $|G_i|\geqslant 2$, and let $G=\bigoplus_{i=1}^{\infty}G_i$. Consider any increasing sequence $(k_n)$ of positive integers, and let
\[
F_n=\bigoplus_{i=1}^{k_n}G_i
\]
with $F_0=\{1_G\}$, and 
\[
H_n=\bigoplus_{i=k_n+1}^{k_{n+1}}G_i,
\]
where to define $H_0$ we take $k_0=0$. The sets $(F_n)$ are naturally identified with subgroups of $G$, where they form a Følner sequence. The groups $H_n$ can be identified with subgroups of $F_{n+1}$ under the map $\varphi_n(h)=(0, h)$.

Now pick a sequence of arbitrary functions $(s_n)$, where $s_n:H_n\to F_n.$ Let $c_{n+1}:H_n\to F_{n+1}$ be given by $c_{n+1}(h)=(s_n(h), h)$. Finally, denote by $C_{n+1}\subset F_{n+1}$ the image of $H_n$ under the map $c_{n+1}$.

We can then define 
\[
X_n=F_n\times C_{n+1}\times C_{n+2}\times \ldots
\]
and endow it with the product topology (making it a Cantor space). We also define homeomorphisms $i_n:X_n\to X_{n+1}$ given by
\[
i_n(f_n, d_{n+1}, d_{n+2}, \ldots )=(f_nd_{n+1}, d_{n+2}, \ldots ).
\]
The space $X_n$ admits a natural action of the group $F_n$: for $g\in F_n$ we let
\[
g.(f_n, d_{n+1}, d_{n+2}, \ldots)=(gf_n, d_{n+1}, d_{n+2}, \ldots).
\]
It is straightforward to see that the actions of $F_n$ on spaces $X_k$ and $X_l$ with $n\leqslant k\leqslant l$ are conjugated by the map
\begin{equation}\label{conj}
i_{l-1}\circ i_{l-2}\circ\ldots \circ i_k,
\end{equation}
which lets us define an action of $F_n$ on the inductive limit $X$ of the sequence $(X_n, i_n)$ (which is canonically homeomorphic to each $X_n$). 

On each $X_n$ we can also consider the product measure, preserved by the action of $F_n$; these measures are also conjugated by the maps (\ref{conj}), so they extend to the limit $X$. Thus, we obtain a topological measure-preserving system $(X, G, \mu)$. We call systems obtained this way \textit{$(C, F)$-systems}.

\begin{thm}\cite[Theorem 2.1]{Danilenko_2001}\label{CFuni}
All $(C, F)$-systems are minimal and uniquely ergodic.
\end{thm}

\begin{thm}\cite[Theorem 0.1]{Danilenko_2006}\label{CFmixing}
    For any sequence $(G_i)$, there exist sequences $(k_n)$ and $(s_n)$ making the system $(X, G, \mu)$ mixing of all orders.
\end{thm}

\begin{thm}\cite[Theorem 0.4]{Danilenko_2006}\label{CFMSJ}
    If the group $G$ is abelian, any $(C, F)$-system which is mixing of all orders has minimal self-joinings of all orders.
\end{thm}
\subsection{IC-rigidity of $(C, F)$-systems}

In this section we require the groups $G_i$ to be abelian. 

To apply Theorem \ref{mainD} to a $(C, F)$-system $(X, G, \mu)$, we are only missing the lack of weakly asymptotic pairs. 
\begin{prop}\label{CFasy}
Let $(X, G, \mu)$ be a mixing $(C, F)$-system, and take points $x, x'\in X$ lying in different orbits of $G$. Then the point $(x, x')\in X\times X$ is generic for the product measure $\mu^{\times 2}$. 
\end{prop}
In particular, by Lemma \ref{lemma:WPgen} the system has no weakly asymptotic pairs.

\begin{proof}
Consider the sequence of averages
\begin{equation}\label{CFaver}
\frac{1}{|F_n|}\sum\limits_{g\in F_n}\delta_{(gx, gx').}
\end{equation}
We prove that it converges to the product measure $\mu^{\times 2}$, which ends the proof. 

For a set $B\subset F_k$, let $[B]_{k}^{\infty}\subset X_k$ denote the set of points whose first coordinate lies in $B$. Notice that for $n\geqslant k$ each element of $F_n$ can be represented uniquely as $f_k\cdot d_{k+1}\cdot \ldots\cdot d_{n}$ for $f_k\in F_k$ and $d_i\in D_i$. Let $[B]_{k}^n\subset F_n$ be the set of elements, for which in this representation we have $f_k\in B$. Let us state a few properties of such cylinder operations; we omit the proofs, which are straightforward algebraic manipulations.

\begin{claim}
For $B\subset F_k$ and any $n\geqslant k$, 
   \[      
   \mu\left([B]_{k}^{\infty}\right)=\frac{|B|}{|F_k|}=\frac{|[B]_{k}^n|}{|F_n|}.
   \]
Additionally, for any $B'\subset F_k$, and any $g\in F_k$ we have
    \begin{gather*}
        [[B]_{k}^n]_{n}^{\infty}=[B]_{k}^{\infty};\\
        [B\cap B']_{k}^{\infty}=[B]_{k}^{\infty}\cap[B']_{k}^{\infty}\qquad\text{and}\qquad [B\cap B']_{k}^{n}=[B]_{k}^{n}\cap[B']_{k}^{n};\\
        [gB]_{k}^{\infty}=g[B]_{k}^{\infty}\qquad\text{and} \qquad[gB]_{k}^n=g[B]_k^n.
    \end{gather*}
\end{claim}
The family of all cylinder sets $[B]_k^{\infty}$ generates the $\sigma$-algebra of Borel sets on $X$. Therefore, to prove that averages (\ref{CFaver}) converge to the product measure it is enough to show that for any $B, B'\subset F_k$ we have
\[
\lim_{n\to\infty}\frac{1}{|F_n|}\#\{g\in F_n: (gx, gx')\in [B]_k^{\infty}\times [B']_{k}^{\infty}\}=\mu([B]_k^{\infty})\cdot \mu( [B']_k^{\infty}).
\]
Denote
\[
x=(f_n, d_{n+1}, \ldots)\qquad\text{and}\qquad x
=(f_n', d_{n+1}', \ldots ).
\]
Then, denoting $t_n=f_n^{-1}f_n'$, we have
\begin{align*}
\#\{g\in F_n: (gx, gx')\in [B]_k^{\infty}\times [B']_k^{\infty}\}&=\#\{g\in F_n: (gf_n, gf_n')\in [B]_k^n\times [B']_k^n\}\\&=\#\{h\in F_n: (h, ht_{n})\in [B]_k^n\times [B']_k^n\}
\end{align*}
using the substitution $h=gf_n$. But this is simply the size of the set $[B]_k^n\cap t_n^{-1}[B']_k^n\subset F_n$. Therefore, by the properties of the cylinder operation, we have
\[
\frac{1}{|F_n|}\#\{g\in F_n: (gx, gx')\in [B]_k^{\infty}\times [B']_k^{\infty}\}=\mu\Big(\big[[B]_k^n\cap t_n^{-1}[B']_k^n\big]_{n}^{\infty}\Big)=\mu\Big([B]_{k}^{\infty}\cap t_n^{-1}[B']_k^{\infty}\Big).
\]
We now use the fact that points $x, x'$ lie on different orbits to prove that the sequence $(t_n)$ escapes every element of the Følner sequence $(F_n)$. Assuming this, by mixing we obtain
\[
\lim_{n\to\infty}\mu([B]_{k}^{\infty}\cap t_n^{-1}[B']_k^{\infty})=\mu([B]_k^{\infty})\cdot \mu([B']_k^{\infty}),
\]
ending the proof.

\begin{claim}
    Either $t_{n+1}\notin F_n$, or $t_{n+1}=t_n$ and $d_{n+1}=d_{n+1}'$.
\end{claim}
\begin{proof}
The point $x$ is represented as $(f_nd_{n+1}, d_{n+2}, \ldots )\in X_{n+1}$, so that
\[
t_{n+1}=(f_nd_{n+1})^{-1}\cdot(f_n'd_{n+1}').
\]
Recall that $d:=d_{n+1}$ and $d':=d_{n+1}'$ are values of the function $c_{n+1}$. By the definition of this function, we either have $d=d'$ or $d^{-1}d'\notin F_n$. Since $t_{n+1}=t_n\cdot d^{-1}d'$ and $t_n\in F_n$, this implies our claim.
\end{proof}
Now one can readily show by induction that for any $N$, either $t_n\notin F_N$ for large enough $n$, or $d_n=d_n'$ for $n> N$. In the latter case, there exists $m\geq 1$ such that $t_n=t_m\in F_m$ for $n\geq m$, which implies $x'=t_m x$, so they are in a single orbit of the action of $G$. Hence, $(t_n)$ must eventually escape every set $F_N$, which is precisely what we required.
\end{proof}

\begin{thm}
    For abelian $G$, any $(C, F)$-action $(X, G, \mu)$ which is mixing of all orders is also IC-rigid.
\end{thm}

\begin{proof}
    We verify all conditions of Theorem \ref{mainD}. Theorem $\ref{CFuni}$ gives minimality and unique ergodicity, mixing of all orders is assumed, and Theorem \ref{CFMSJ} gives minimal self-joinings of all orders. Finally, Lemma \ref{lemma:WPgen} together with Proposition \ref{CFasy} shows that there are no weakly asymptotic pairs. 
\end{proof}

By Theorem \ref{CFmixing}, we obtain Theorem \ref{mainF} as an immediate corollary. 

Finally, we remark that, by Corollary \ref{cor:topmsj}, such systems have topological minimal self-joinings of all orders.

\printbibliography
    
\end{document}